\documentclass[11pt]{amsart}
\usepackage[utf8]{inputenc}

\usepackage{amsmath}
\usepackage{amsthm}
\usepackage{amsfonts}
\usepackage{amssymb}
\usepackage{mathtools}
\usepackage{stmaryrd}

\usepackage{graphicx}
\usepackage{placeins}

\usepackage{hyperref}

\newcommand{\C}{{\mathbb C}}

\newcommand{\Z}{{\mathbb Z}}
\newcommand{\I}{{\mathbb I}}
\newcommand{\U}{{\mathbf U}}
\newcommand{\K}{{\widetilde K}}

\newcommand{\hbracket}[2]{%
\begin{bmatrix}%
h; {#1}\\{#2}%
\end{bmatrix}%
}
\newcommand{\of}[1]{{\left({#1}\right)}}
\newcommand{\floor}[1]{{\left\lfloor{#1}\right\rfloor}}

\DeclareMathOperator{\ad}{ad}
\DeclareMathOperator{\End}{End}
\DeclareMathOperator{\Aut}{Aut}

\newcommand{\blue}[1]{{\color{blue}#1}}

\newcommand{\vs}{\varsigma}

\makeatletter
\newcommand{\@spaceifnotempty}[1]{%
    \if\relax\detokenize{#1}\relax
    \else\ \fi
}

\newcommand{\@case}[3]{\text{if }{#2}\text{\@spaceifnotempty{#1}{#1}{#3}}}

\newcommand{\case}[2][]{{\@case{#1}{#2};}}
\newcommand{\lastcase}[2][]{{\@case{#1}{#2}.}}

\makeatother

\newcommand{\Iblack}{\I_{\bullet}}
\newcommand{\Iwhite}{\I_{\circ}}
\newcommand{\Isplit}{\Iwhite^\text{split}}
\newcommand{\wblack}{{w_\bullet}}

\newcommand{\Y}{\check{E}}

\newenvironment{bbmatrix}
    {\left\llbracket\begin{matrix}}
    {\end{matrix}\right\rrbracket}

\DeclareMathOperator{\SU}{SU}
\newcommand{\Uqisl}[1]{{\U_{q_i}\of{\mathfrak{sl}_{#1}}}}
\newcommand{\Uqsl}[1]{{\U_q\of{\mathfrak{sl}_{#1}}}}
\newcommand{\Ui}[1][i]{{\U_{#1}}}

\newcommand{\vbar}[1]{\left.{#1}\right|}
\newcommand{\texti}{\texorpdfstring{$\imath$}{ı}}

\usepackage{tikz-cd}
\usetikzlibrary{calc}
\usetikzlibrary{decorations.pathmorphing}

\tikzset{curve/.style={settings={#1},to path={(\tikztostart)
    .. controls ($(\tikztostart)!\pv{pos}!(\tikztotarget)!\pv{height}!270:(\tikztotarget)$)
    and ($(\tikztostart)!1-\pv{pos}!(\tikztotarget)!\pv{height}!270:(\tikztotarget)$)
    .. (\tikztotarget)\tikztonodes}},
    settings/.code={\tikzset{quiver/.cd,#1}
        \def\pv##1{\pgfkeysvalueof{/tikz/quiver/##1}}},
    quiver/.cd,pos/.initial=0.35,height/.initial=0}

\tikzset{tail reversed/.code={\pgfsetarrowsstart{tikzcd to}}}
\tikzset{2tail/.code={\pgfsetarrowsstart{Implies[reversed]}}}
\tikzset{2tail reversed/.code={\pgfsetarrowsstart{Implies}}}
\tikzset{no body/.style={/tikz/dash pattern=on 0 off 1mm}}

\newtheorem{theorem}{Theorem}[section]
\newtheorem{lemma}[theorem]{Lemma}

\newtheorem{proposition}[theorem]{Proposition}

\theoremstyle{definition}
\newtheorem{definition}{Definition}[section]

\theoremstyle{remark}
\newtheorem{remark}{Remark}[section]

\begin{document}

\title[A new approach to the {\texti}Serre and Serre-Lusztig relations]{A new approach to the {\texti}Serre and Serre-Lusztig relations for {\texti}quantum groups}
\author{Zachary Carlini}
\address{Department of Mathematics, University of Virginia, Charlottesville, VA 22904}
\email{zic4zfr@virginia.edu}

%\subjclass[2010]{Primary 17B37,  16E60, 18E30.}

\maketitle

\begin{abstract}
    We give a new, conceptual proof of the {\texti}Serre and Serre-Lusztig relations for {\texti}quantum groups. The key to our approach is a new formula for the comultiplication of the \texti-divided powers, which allows us to reformulate the relations in terms of the adjoint action. We then obtain a proof using properties of the adjoint representation. The flexibility of this approach allows us to establish a more general family of relations which seem difficult to establish otherwise.
\end{abstract}

\setcounter{tocdepth}{1}
\tableofcontents

\section{Introduction}

Associated to a Satake diagram (\(\I = \Iwhite \cup \Iblack, \tau\)) of finite type (or Kac-Moody type), one defines a quantum symmetric pair \((\U, \U^\imath_\vs)\) \cite{Let99} (also see \cite{Ko14}) that can be seen as a quantum analog of the associated classical symmetric pair. While the coideal subalgebra \(\U^\imath_\vs\) of $\U$ is not in general a quantum group, every quantum group \(\U\) appears in a diagonal quantum symmetric pair \(\of{\U \otimes \U, \U}\); cf. \cite{BW18b}. Thus, the algebras of the form \(\U^\imath_\vs\) serve as a generalization of quantum groups which we will refer to as {\texti}quantum groups. We say that an \(\imath\)quantum group is \emph{quasi-split} if \(\Iblack = \emptyset\), and \emph{split} if, additionally, \(\tau\) is the trivial automorphism. %Recent developments (cf. \cite{BW18a}, \cite{CLW21}, \cite{CLW21b}, and references therein) have generalized much of the theory of quantum groups to {\texti}quantum groups.

A presentation of {\texti}quantum groups of finite type was obtained by Letzter \cite{Let02, Let03}, which involved Serre type relations with mysterious \(q\)-powers and inhomogenous terms; this was extended by Kolb to $\imath$quantum groups of certain Kac-Moody type \cite{Ko14}. A conceptual presentation for quasi-split {\texti}quantum groups of arbitrary Kac-Moody type was given in \cite{CLW21}. This formulation relies on the \texti-divided powers introduced in \cite{BW18a} (see also \cite{BeW18}). The simple expression for the {\texti}Serre relations given in \cite{CLW21} is formally the same as the usual \(q\)-Serre relations, with the \texti-divided powers playing the role of the Lusztig divided powers. 
%However, the proof was long and calculational, and it relied on highly nontrivial \(q\)-combinatoric identities. 
Casper, Kolb, and Yakimov found a different Serre presentation in the quasi-split case, expressed using some special polynomials instead of the {\texti}-divided powers \cite{CKY21}. This approach was later extended to the general Kac-Moody type in \cite{KY21}. 

However, the {\texti}Serre relations via {\texti}-divided powers are desirable for several applications. First, they are formulae in the integral form of \(\U^\imath_\vs\) for suitable $\vs$, which are expected to play a basic role in categorification. Moreover, this form of the Serre presentation was essential for establishing the {\texti}Hall algebra realization of quasi-split $\imath$quantum groups \cite{LW20}. In addition, it allows for natural extension to the higher order Serre-Lusztig relations using $\imath$divided powers \cite{CLW21b}, which has further led to relative braid group formulas of $\imath$quantum groups of Kac-Moody type; see \cite{CLW21b, Z22}. Finally, they are instrumental in the study of $\imath$quantum groups at roots of 1 \cite{BS22}.

The goal of this paper is to give a self-contained and conceptual new proof of the {\texti}Serre relations and Serre-Lusztig relations of minimal degree; the original proofs in \cite{CLW21, CLW21b} were long and computational. We develop a new approach based on the adjoint operator. This approach has also led to new relations which seem difficult to establish otherwise.

It is well known that the \(q\)-Serre relations for a quantum group can be expressed compactly using the adjoint action; cf. \cite{Jan96}. In the present paper, we establish an analogous formulation for the {\texti}Serre and Serre-Lusztig relations of minimal degree. To show that this adjoint operator formulation is equivalent to the formulation in \cite{CLW21, CLW21b}, we prove a new comultiplication formula for the {\texti}-divided powers, building on \cite{CW23}. 

We then give a direct proof of this adjoint operator reformulation and hence obtain self-contained new proofs of the {\texti}Serre and Serre-Lusztig relations of minimal degree. Specifically, we use \cite[Theorems 2.10, 3.6]{BeW18}, which assert that for the split rank \(1\) {\texti}quantum group with parameter \(\vs = q^{-1}\), the $\imath$divided power \(B_{\overline n}^{(n + 1)}\) annihilates the simple \(\Uqsl2\)-module of highest weight \(q^n\). After slightly strengthening this result (see Lemma \ref{thm:BMinimalPolynomial} and Lemma~ \ref{thm:strongBMinimalPolynomial}), we apply it to the adjoint representation to obtain quick proofs of the reformulated {\texti}Serre and Serre-Lusztig relations. 

The paper is organized as follows.
In Section~\ref{sec:prelim}, we recall some of the basics of $\imath$quantum groups, including  $\imath$-divided powers and comultiplication. In Section~\ref{sec:reformulation}, we obtain a new comultiplication formula for the $\imath$-divided powers and apply it to obtain new formulations of the $\imath$Serre relations and Serre-Lusztig relations in split $\imath$quantum groups via the quantum adjoint action. Finally, in Section~\ref{sec:newproof}, we use these reformulations to give new proofs of the $\imath$Serre and Serre-Lusztig relations, and we establish some new relations using the same methods.

\vspace{2mm}

\paragraph{\bf Acknowledgements.}
The author would like to thank Weiqiang Wang for his incredibly helpful advice and discussion, without which this paper would not have been possible. The author's undergraduate research is supported by Wang's NSF grant (DMS-2001351).

\section{The preliminaries}
\label{sec:prelim}

\paragraph{\bf Quantum groups.}
Let \((\I, \cdot)\) be a Cartan datum. For \(i, j \in \I\), we set \(a_{ij} = 2\frac{i \cdot j}{i \cdot i}\) and \(\epsilon_i = \frac{i \cdot i}{2}\). A root datum of type \((\I, \cdot)\) consists of finitely generated free abelian groups \(X\) and \(Y\), a perfect pairing \(\langle\cdot, \cdot\rangle : Y \times X \to \Z\), and elements \(\alpha_i \in X\) and \(h_i \in Y\) associated to every \(i \in \I\) such that \(\langle h_i, \alpha_j\rangle = a_{ij}\) for every \(i, j \in \I\); cf. \cite{Lus93}. We will assume that the root datum is \(X\)-regular and \(Y\)-regular, so \(\{\alpha_i : i \in \I\}\) is linearly independent in \(X\) and \(\{h_i : i \in \I\}\) is linearly independent in \(Y\).

For \(i \in \I\), we define \(q_i \in \C(q)\) by \(q_i = q^{\epsilon_i}\). We define the quantum integers and quantum factorial as follows: \begin{equation}
    [n]_i = \frac{q_i^n - q_i^{-n}}{q_i - q_i^{-1}}; \qquad [n]_i^! = [n]_i [n - 1]_i \dots [1]_i.
\end{equation} 

Recall that the Drinfeld-Jimbo quantum group \(\U\) is the \(\C(q)\)-algebra with generators \(E_i, F_i,\) and \(K_h\) for all \(i \in \I\) and \(h \in Y\), subject to the relations \begin{align}
    K_0 = 1, \qquad K_h K_{h'} &= K_{h + h'}; \tag{R1}\label{eq:krelations} \\
    K_h E_j K_{-h} &= q^{\langle h, \alpha_j \rangle} E_j; \tag{R2}\label{eq:erelation} \\
    K_h F_j K_{-h} &= q^{-\langle h, \alpha_j \rangle} F_j; \tag{R3}\label{eq:frelation} \\
    E_i F_j - F_j E_i &= \delta_{i j} \frac{\K_i - \K_i^{-1}}{q_i - q_i^{-1}}; \tag{R4}\label{eq:efcommutator} \\
    \sum_{r + s = 1 - a_{ij}}  E_i^{(r)} E_j E_i^{(s)} &= 0; \tag{R5}\label{eq:eqSerre} \\
    \sum_{r + s = 1 - a_{ij}} (-1)^s   F_i^{(r)} F_j F_i^{(s)} &= 0, \tag{R6}\label{eq:fqSerre}
\end{align} 
for all \(i, j \in \I\), \(h, h' \in Y\), where $\K_i = K_{h_i}^{\epsilon_i}$, and Lusztig's divided powers are defined by \(E_i^{(n)} = E_i^n/ [n]_i^!, F_i^{(n)} =F_i^n/ [n]_i^!\).

We recall the Hopf algebra structure on \(\U\); cf. \cite{Jan96, Lus93}. The antipode \(S : \U \to \U\) is the unique antiautomorphism satisfying 
\[
    S(K_h) = K_{-h}; \qquad S(E_i) = -\K_i^{-1} E_i; \qquad S(F_i) = -F_i \K_i,
\]
for all \(i \in \I\) and \(h \in Y\), 
and the comultiplication \(\Delta : \U \to \U \otimes \U\) is the unique algebra homomorphism satisfying \[
    \Delta(K_h) = K_h \otimes K_h; \quad
    \Delta(E_i) = E_i \otimes 1 + \K_i \otimes E_i; \quad
    \Delta(F_i) = F_i \otimes \K_i^{-1} + 1 \otimes F_i.
\]

\paragraph{\bf The {\texti}quantum groups.}
Let \(\tau\) be an involution of \((\I, \cdot)\). We assume that \(\tau\) extends to an involution of \(X\) and an involution of \(Y\) such that the pairing \(\langle \cdot, \cdot \rangle\) is preserved by \(\tau\). Let \(\Iblack \subset \I\) be a subdatum of finite type. We denote \(\Iwhite = \I \setminus \Iblack\). Let \(\wblack\) be the longest element in the Weyl group of \(\Iblack\). We denote 
\[
\Isplit = \{i \in \Iwhite : \tau i = i = \wblack i\}.
\]
%Let \(\rho^\vee_\bullet\) be the half some of all positive coroots associated to \(\Iblack\). 
The pair \((\I =\Iblack \cup \Iwhite, \tau)\) is required to be admissible in the sense of \cite[Definition 2.3]{Ko14}.
%: \begin{itemize}
 %   \item \(\tau\of{\Iblack} = \Iblack\);
 %   \item \(\omega_\bullet j = -\tau j\) for all \(j \in \Iblack\);
 %   \item \(\langle \rho^\vee_\bullet, \alpha_i \rangle \in \Z\) for all \(i \in \Iwhite\) with \(\tau i = i\).
%\end{itemize}
% We will always assume the pair \((\I =\Iblack \cup \Iwhite, \tau)\) to be admissible.

Fix a choice of parameters \(\vs = \of{\vs_i}_{i \in \Iwhite}\) satisfying the constraints $\vs_i =\vs_{\tau i}$ if $a_{i, \tau i}=0$. Following \cite{Let99, Ko14}, we define the {\texti}quantum group \(\U^\imath_\vs\) to be the subalgebra of \(\U\) generated by all \(E_j\) and \(F_j\) for \(j \in \Iblack\), all \(K_h\) for \(h \in Y\) with \(\tau h = -\wblack h\), and all \[
    B_i \coloneqq F_i + \vs_i T_\wblack(E_{\tau i})\K_i^{-1}
\]
for \(i \in \Iwhite\).
%(note that our definition requires all \(s_i = 0\), cf. \cite{Ko14}). 
Here \(T_\wblack\) corresponds to \(T_{\wblack, +1}''\) in \cite[Ch. 37]{Lus93}.

For \(i \in \Isplit\), the subalgebra of \(\U\) generated by \(E_i, F_i, \K_i, \text{ and } \K_i^{-1}\) will be denoted \(\Ui\). The operator \(T_\wblack\) acts trivially on this subalgebra, so in particular we have \(B_i = F_i + \vs_i E_i K_i^{-1}\). We will introduce the following additional notation: 
\begin{align}
    \Y_i = \vs_i E_i \K_i^{-1}; \qquad
    \hbracket{a}{n}_i = \prod_{i = 1}^n \frac{q_i^{4a + 4i - 4}\K_i^{-2} - 1}{q_i^{4i} - 1}.
\end{align}
We will also define the divided powers of \(\Y_i\): \(\Y_i^{(n)} = \Y_i^n/[n]_i^!\).

When \(i \in \Isplit\), we can identify the polynomial subalgebra of \(\U\) generated by \(B_i\) with the split rank \(1\) {\texti}quantum group. Following \cite{CLW21} (which slightly generalizes the definition in \cite{BW18a, BeW18} to allow arbitrary \(\vs_i\)) we define the \texti-divided powers, depending on a choice of parity \(\overline p \in \{\overline 0, \overline 1\}\):
\begin{align}
        B_{i, \overline 0}^{(n)} &= \frac{1}{[n]!}\begin{cases}
            B_i\prod_{j=1}^k (B_i^2-\blue{q_i\vs_i}[2j]_{i}^2) & \case{n = 2k + 1} \\
            \prod_{j=1}^{k} (B_i^2-\blue{q_i\vs_i}[2j-2]_{i}^2) & \case{n = 2k}
        \end{cases} \label{eq:evenDividedPowerDefinition} \\
        B_{i, \overline 1}^{(n)} &= \frac{1}{[n]!}\begin{cases}
            B_i\prod_{j=1}^k (B_i^2-\blue{q_i\vs_i}[2j-1]_{i}^2) & \case{n = 2k + 1} \\
            \prod_{j=1}^k (B_i^2-\blue{q_i\vs_i}[2j-1]_{i}^2) & \lastcase{n = 2k}
        \end{cases} \label{eq:oddDividedPowerDefinition}
    \end{align}
These \texti-divided powers are related to the special case where \(\vs_i = q_i^{-1}\) by a certain rescaling automorphism.
\begin{definition}\label{def:rescalingauto}
    For \(\lambda \in \C(q)^*\), define \(\xi_\lambda\) to be the unique endomorphism of \(\U\) which satisfies \begin{align*}
        \xi_\lambda(E_i) = \lambda E_i; \quad
        \xi_\lambda(F_i) = \lambda^{-1} F_i; \quad
        \xi_\lambda(K_h) = K_h,
    \end{align*}
    for all \(i \in \I\), \(h \in Y\).
\end{definition}

We list several key properties, which the reader can verify directly.
\begin{enumerate}
    \item For any \(\lambda \in \C(q)^*\), \(\xi_\lambda\) is a Hopf algebra automorphism.
    \item The map \(C(q)^* \to \Aut (\U) \) that sends \(\lambda\) to \(\xi_\lambda\) is a group homomorphism.
    \item For \(i \in \Isplit\) and \(\lambda \in \C(q)^*\), \(\xi_\lambda \big(\Y_i^{(n)}\big) = \lambda^{n}\Y_i^{(n)}\), and \(\xi_\lambda \big(F_i^{(n)}\big) = \lambda^{-n}F_i^{(n)}\).
    \item For \(i \in \Isplit\) and \(u \in \Ui\), \(S^2(u) = \xi_{q_i}^{-2}\of{u}\).
    \item For \(h \in Y\), \(i \in \Isplit\), and \(u \in \Ui\), \(K_h u K_{-h} = \xi_{q^{\langle h, \alpha_i \rangle}}(u)\).
    \item For \(i \in \Isplit\), \(n \geq 0\), and \(\overline p \in \{\overline 0, \overline 1\}\), set \(\widetilde{B_{i, \overline p}^{(n)}}\) to be the \texti-divided powers with parameter \(q_i^{-1}\). By working in an extension of \(\C(q)\), we may set \(z = \sqrt{q_i \vs_i}\). Then \(\xi_z\bigg(\widetilde{B_{i, \overline p}^{(n)}}\bigg) = z^{-n} {B_{i, \overline p}^{(n)}}\).
\end{enumerate}

Explicit formulae have been found for the \(\imath\)-divided powers \(B_{i, \overline p}^{(n)}\) in terms of the Chevalley generators. We will only need the even parity case, given (up to the application of an appropriate rescaling automorphism; cf. \cite{CW23}) in \cite[Proposition 2.7]{BeW18}:
\begin{equation}\label{eq:idvidedpowers}
    B_{i, \overline 0}^{(n)} = \sum_{\substack{a + 2c \leq n \\ a, c \geq 0}} k_{i, n, c, a} F_i^{(n - 2c - a)}\hbracket{1 - c + \floor{\frac{n - 1}{2}}}{c}_i\Y_i^{(a)},
\end{equation}
where \(k_{i, n, a, c}\) is given by \[
    k_{i, n, a, c} = \begin{cases}
        (-1)^c q_i^{3c + a(n - 2c - a)}\blue{(q_i \vs_i)^c} & \case[is even]{n} \\
        (-1)^c q_i^{c + a(n - 2c - a)}\blue{(q_i \vs_i)^c} & \lastcase[is odd]{n}
    \end{cases}
\]

Our goal is to study the adjoint action of \(B_{i, \overline{1-n}}^{(n)}\) on \(\U\). We shall begin by studying \(\Delta\of{B_{i, \overline{1-n}}^{(n)}}\). For any \(i \in \Isplit\), define \(T_{i, n, r}\) to be the unique elements of \(\Ui\) satisfying \begin{equation}\label{eq:comultT}
    \Delta\of{B_{i, \overline{1 - n}}^{(n)}} = \sum_{r + s = n} B_{i, \overline{1 - n}}^{(s)} \otimes T_{i, n, r}.
\end{equation}

From \cite[Theorem 4.2]{CW23} and \cite[Theorem 5.1]{CW23}, %by applying the appropriate rescaling automorphism,
we have
\begin{equation}\label{eq:Texplicit}
    T_{i, n, r} = \sum_{\substack{a + 2c \leq r \\ a, c \geq 0}} t_{i, n, r, c, a} \Y_i^{(a)} \hbracket{-\floor{\frac{r - 1}{2}}}{c}_i \K_i^{r - n} F_i^{(r - 2c - a)},
\end{equation}
with \(t_{i, n, r, c, a}\) given by \[
    t_{i, n, r, c, a} = q_i^{\binom{2c + 1}{2} + (r-2c)(r-n) - a(r - 2c - a)}\blue{(q_i \vs_i)^c}.
\]

In this paper, we will establish a more compact formula for \(T_{i, n, r}\) (and thus for $\Delta\of{B_{i, \overline{1 - n}}^{(n)}}$)
and use that to find a formula for the adjoint action of \(B_{i, \overline{1-n}}^{(n)}\) on \(\U\).

\section{Reformulation of {\texti}Serre and Serre-Lusztig relations}
\label{sec:reformulation}

\subsection{New comultiplication formula for \texti-divided powers}

We begin by seeking a new expression for \(T_{i, n, r}\) different from \eqref{eq:Texplicit}. 

\begin{lemma}\label{thm:antipodes}
For any \(i \in \Isplit\), the following identities hold: \begin{align}
    S\of{F_i^{(n)}} &= (-1)^n q_i^{2\binom{n + 1}{2}} \K_i^n F_i^{(n)}; \label{eq:fantipode} \\
    S\of{\Y_i^{(n)}} &= (-1)^n q_i^{2 \binom{n}{2}} \Y_i^{(n)} \K_i^n \label{eq:eantipode}; \\
    S\of{\hbracket{a}{n}_i} &= (-1)^n q_i^{2n(n + 2a - 1)} \K_i^{2n}\hbracket{1 - n - a}{n}_i \label{eq:hbracketantipode}.
\end{align}
\end{lemma}
\begin{proof}
First, we can compute \[
    S\of{F_i^{(n)}} = \frac{1}{[n]_i^!}\of{-F_i \K_i}^n = \frac{(-1)^n q_i^{2\binom{n + 1}{2}}}{[n]_i^!} \K_i^n F_i^n = (-1)^n q_i^{2\binom{n + 1}{2}} \K_i^n F_i^{(n)}.
\]
Next, we can observe \[
    S\of{\Y_i} = S\of{\vs_i E_i \K_i^{-1}} = -\vs_i E_i = - \Y_i \K_i,
\]
from which \eqref{eq:eantipode} follows by an argument analogous to the proof of \eqref{eq:fantipode}. Finally, \begin{align*}
    S\of{\hbracket{a}{n}_i} &= \prod_{i = 1}^n \frac{q_i^{4a + 4i - 4}\K_i^{2} - 1}{q_i^{4i} - 1} \\
    &= (-1)^n q_i^{4na - 4n + 4\binom{n + 1}{2}} \K_i^{2n} \prod_{i = 1}^n \frac{q_i^{-4a - 4i + 4} \K_i^{-2} - 1}{q_i^{4i} - 1} \\
    &= (-1)^n q_i^{2n(n + 2a - 1)} \K_i^{2n} \prod_{i = 1}^n \frac{q_i^{-4a - 4n + 4i} \K_i^{-2} - 1}{q_i^{4i} - 1} \\
    &= (-1)^n q_i^{2n(n + 2a - 1)} \K_i^{2n} \hbracket{1 - n - a}{n}_i.
\end{align*}
The lemma is proved. 
\end{proof}

Recall the automorphism \(\xi_{q_i}\) given by Definition \ref{def:rescalingauto}. We improve \cite[Theorem 4.2]{CW23} and \cite[Theorem 5.1]{CW23} (also see \eqref{eq:Texplicit}) as follows.

\begin{theorem}\label{thm:sformula}
    For \(i \in \Isplit\) and \(n \geq 0\), \[
        \Delta\of{B_{i, \overline{1 - n}}^{(n)}} = \sum_{r + s = n} (-1)^r B_{i, \overline{1 - n}}^{(s)} \otimes \K_i^{-n} \xi_{q_i}^{n + 1}\Big(S\big(B_{i, \overline 0}^{(r)}\big)\Big).
    \]
\end{theorem}
\begin{proof}
    From equation \eqref{eq:idvidedpowers}, we have \begin{equation*}
        B_{i, \overline 0}^{(r)} = \sum_{\substack{a + 2c \leq r \\ a, c \geq 0}} 
        k_{i, r, c, a}
        F_i^{(r - 2c - a)}
        \hbracket{1 - c + \floor{\frac{r - 1}{2}}}{c}_i
        \Y_i^{(a)}.
    \end{equation*}
    By Lemma \ref{thm:antipodes}, we can compute
    \begin{equation*}
        S\of{B_{i, \overline 0}^{(r)}} = \sum_{\substack{a + 2c \leq r \\ a, c \geq 0}} 
        k'_{i, r, c, a}
        \Y_i^{(a)} \K_i^a
        \K_i^{2c}\of{\hbracket{-\floor{\frac{r - 1}{2}}}{c}_i}
        \K_i^{r - 2c - a} F_i^{(r - 2c - a)},
    \end{equation*}
    with \[
        k'_{i, r, c, a} = (-1)^{r - 3c}q_i^{2\binom{a}{2} + 2\binom{r - 2c - a + 1}{2} + 2c\of{c + 2\of{1 - c + \floor{\frac{r - 1}{2}}} - 1}}k_{i, r, c, a}.
    \]
    We can simplify \[
        k'_{i, r, c, a} = (-1)^r q_i^{\binom{2c + 1}{2} + a^2 - 2a + (r - a + 1)(r - 2c)}\blue{(q_i \vs_i)^c}.
    \]
    We can then pull out \(\K_i^n\) on the left to obtain
    \begin{equation*}
        S\of{B_{i, \overline 0}^{(r)}} = \K_i^{n} \sum_{\substack{a + 2c \leq r \\ a, c \geq 0}} 
        k_{i, n, r, c, a}''
        \Y_i^{(a)}
        \of{\hbracket{-\floor{\frac{r - 1}{2}}}{c}_i}
        \K_i^{r - n}
        F_i^{(r - 2c - a)},
    \end{equation*}
    where \[
        k_{i, n, r, c, a}'' = q_i^{-2na}k_{i, r, c, a}'.
    \]
    But we can compute \begin{align*}
        k_{i, n, r, c, a}'' &= (-1)^r q_i^{\binom{2c + 1}{2} + a^2 - 2a + (r - a + 1)(r - 2c) - 2na}\blue{(q_i \vs_i)^c} \\
        &= (-1)^r q_i^{\binom{2c + 1}{2} - a(r - 2c - a) + (r - 2c)(r - n) + (n + 1)(r - 2c - 2a)}\blue{(q_i \vs_i)^c} \\
        &= (-1)^r q_i^{(n + 1)(r - 2c - 2a)}t_{i, n, r, c, a},
    \end{align*}
    so by applying \(\xi_{q_i}^{n + 1}\), we obtain \begin{equation*}
        \xi_{q_i}^{n + 1}\Big(S\big(B_{i, \overline 0}^{(r)}\big)\Big) = (-1)^r \K_i^{n} T_{i, n, r},
    \end{equation*}
    and the result follows.
\end{proof}

\subsection{The adjoint operator}
For \(u, v \in \U\), if \(\Delta(u) = \sum_{t} u_t \otimes u^t\), the adjoint action of \(u\) on \(v\) is defined by \[
    \ad(u)(v) = \sum_{t} u_t v S\of{u^t}.
\]
Then \(\ad : \U \to \End(\U)\) is an algebra homomorphism.

\begin{proposition}
\label{thm:oppositeparityadformula}
    For \(i \in \Isplit\), \(n \geq 0\), and \(u \in \U\), we have
    \[
        \ad\of{B_{i, \overline{1 - n}}^{(n)}}(u) = \sum_{r + s = n}(-1)^r B_{i, \overline{1 - n}}^{(s)} u \xi_{q_i}^{n - 1}\of{B_{i, \overline 0}^{(r)}} \K_i^n.
    \]
\end{proposition}

\begin{proof}
    By definition, we have \[
        \ad\of{B_{i, \overline{1 - n}}^{(n)}}(u) = \sum_{r + s = n} B_{i, \overline{1 - n}}^{(s)} u S\of{T_{i, n, r}}.
    \]
    Then by Theorem \ref{thm:sformula}, \[
        \ad\of{B_{i, \overline{1 - n}}^{(n)}}(u) 
        = \sum_{r + s = n} B_{i, \overline{1 - n}}^{(s)} u S\bigg((-1)^r \K_i^{-n} \xi_{q_i}^{n + 1}\bigg(S\Big(B_{i, \overline 0}^{(r)}\Big)\bigg)\bigg),
    \]
    which simplifies to \[
        \ad\of{B_{i, \overline{1 - n}}^{(n)}}(u) 
        = \sum_{r + s = n} (-1)^r B_{i, \overline{1 - n}}^{(s)} u \xi_{q_i}^{n - 1}\of{B_{i, \overline 0}^{(r)}}\K_i^{n}.
    \]
    This completes the proof. 
\end{proof}

\subsection{New formulation of {\texti}Serre and Serre-Lusztig relations}

By a standard result for quantum groups (see \cite[4.18]{Jan96}), for any \(i, j \in \I\), \begin{equation}\label{eq:qSerreAdjointFormula}
    \ad\of{F_i^{(1 - a_{ij})}}\of{F_j\K_j} = \sum_{r + s = 1 - a_{ij}} (-1)^r F_i^{(s)}F_jF_i^{(r)}\K_j \K_i^{1 - a_{ij}}.
\end{equation}
Using Proposition \ref{thm:oppositeparityadformula}, we can prove an {\texti}quantum analog.

\begin{proposition}
\label{thm:iSerreAdjoint}
Let \(i \in \Isplit\) and \(j \in \Iwhite\) and $i\neq j$. Then
\[
    \ad\of{B_{i, \overline{a_{ij}}}^{(1 - a_{ij})}}\of{B_j\K_j} = \sum_{r + s = 1 - a_{ij}} (-1)^r B_{i, \overline{a_{ij}}}^{(s)}B_jB_{i, \overline 0}^{(r)} \K_j \K_i^{1 - a_{ij}}.
\]
\end{proposition}
\begin{proof}
    By Proposition \ref{thm:oppositeparityadformula}, \[
        \ad\of{B_{i, \overline{a_{ij}}}^{(1 - a_{ij})}}\of{B_j \K_j}
        = \sum_{r + s = 1 - a_{ij}} (-1)^r B_{i, \overline{a_{ij}}}^{(s)} B_j \K_j  \xi_{q_i}^{-a_{ij}}\of{B_{i, \overline 0}^{(r)}} \K_i^{1 - a_{ij}}.
    \]
    Moving the \(\K_j\) term to the right, we obtain \[
        \ad\of{B_{i, \overline{a_{ij}}}^{(1 - a_{ij})}}\of{B_j \K_j}
        = \sum_{r + s = 1 - a_{ij}} (-1)^r B_{i, \overline{a_{ij}}}^{(s)} B_j B_{i, \overline 0}^{(r)} \K_j  \K_i^{1 - a_{ij}}.
    \]
    This completes the proof. 
\end{proof}

Proposition~\ref{thm:iSerreAdjoint} has the following immediate consequence, which proves a conjecture of Wang (via private communication). 

\begin{theorem}  \label{cor:eq}
For  \(i \in \Isplit\) and \(j \in \Iwhite\) with $i\neq j$, the following two relations are equivalent: 
\begin{enumerate}
    \item 
    $\ad\of{B_{i, \overline{a_{ij}}}^{(1 - a_{ij})}}\of{B_j\K_j} =0$
     in $\U$; 
    \item $\sum_{r + s = 1 - a_{ij}} (-1)^r B_{i, \overline{a_{ij}}}^{(s)}B_jB_{i, \overline 0}^{(r)}=0$ in $\U^\imath$.
\end{enumerate}
\end{theorem}

The relation in Theorem~\ref{cor:eq}(2) is known as $\imath$Serre relation in $\U^\imath$, which was first formulated and established in \cite{CLW21}.

\begin{remark}
    For \emph{split} {\texti}quantum groups, the {\texti}Serre relations are all the defining relations; cf. \cite{CLW21}. If we define \(\widehat \U\) to be the algebra generated by \(E_i\), \(F_i\), \(\K_i\), and \(\K_i^{-1}\) for all \(i \in \I\), where we impose relations \eqref{eq:krelations}-\eqref{eq:efcommutator} but not the \(q\)-Serre relations, we can define \(\widehat \U^\imath_\vs\) to be the subalgebra generated by \(F_i + \vs_i E_i \K_i^{-1}\) for \(i \in \I\). We then have a canonical projection \(\pi : \widehat \U \to \U\) that restricts to a projection \(\vbar{\pi}_{\widehat \U^\imath_\vs} \to \U^\imath_\vs\). The proof of Proposition~ \ref{thm:iSerreAdjoint} still applies in \(\widehat U\), and Proposition~ \ref{thm:iSerreAdjoint} and \cite[Theorem 3.1]{CLW21} together imply that the kernel of \(\vbar{\pi}_{\widehat \U^\imath_\vs}\) is equal to the intersection of \(\widehat \U^\imath_\vs\) with the two-sided ideal in \(\widehat \U\) generated by \(\left\{\ad\of{B_{i, \overline{a_{ij}}}^{(1 - a_{ij})}}\of{B_j \K_j} : i \not= j \in \I \right\}\).
\end{remark}

Proposition \ref{thm:iSerreAdjoint} can be strengthened as follows.
\begin{proposition}\label{thm:iSerreLusztigAdjoint}
For \(n \geq 1 \), \(i \in \Isplit\), and \(j \in \Iwhite\), 
\[
    \ad\of{B_{i, \overline{n a_{ij}}}^{(1 - n a_{ij})}}\of{B_j^n\K_j^n} = \sum_{r + s = 1 - n a_{ij}} (-1)^r B_{i, \overline{n a_{ij}}}^{(s)}B_j^nB_{i, \overline 0}^{(r)} \K_j^n \K_i^{1 - n a_{ij}}.
\]
\end{proposition}
\begin{proof}
    By Proposition \ref{thm:oppositeparityadformula}, \[
        \ad\of{B_{i, \overline{n a_{ij}}}^{(1 - n a_{ij})}}\of{B_j^n\K_j^n}
        = \sum_{r + s = 1 - n a_{ij}} (-1)^r B_{i, \overline{n a_{ij}}}^{(s)} B_j^n\K_j^n  \xi_{q_i}^{-n a_{ij}}\of{B_{i, \overline 0}^{(r)}} \K_i^{1 - n a_{ij}}.
    \]
    Moving the \(\K_j^n\) term to the right, we obtain \[
        \ad\of{B_{i, \overline{n a_{ij}}}^{(1 - n a_{ij})}}\of{B_j^n\K_j^n}
        = \sum_{r + s = 1 - n a_{ij}} (-1)^r B_{i, \overline{n a_{ij}}}^{(s)} B_j^n B_{i, \overline 0}^{(r)} \K_j^n  \K_i^{1 - n a_{ij}}.
    \]
    This finishes the proof.
\end{proof}

The following theorem follows directly from Proposition~\ref{thm:iSerreLusztigAdjoint}. 

\begin{theorem}  \label{cor:eq2}
For \(n \geq 1 \), \(i \in \Isplit\) and \(j \in \Iwhite\) with $i\neq j$, the following two relations are equivalent:
\begin{enumerate}
    \item \(\ad\of{B_{i, \overline{n a_{ij}}}^{1 - na_{ij}}}\of{B_j^n \K_j^n} = 0\) in \(\U\);
    \item \(\sum_{r + s = 1 - n a_{ij}} (-1)^r B_{i, \overline{n a_{ij}}}^{(s)}B_j^nB_{i, \overline 0}^{(r)} = 0\) in \(\U^\imath_\vs\).
\end{enumerate}
\end{theorem}

The relation in Theorem~\ref{cor:eq2}(2) is known as Serre-Lusztig (i.e., higher order $\imath$Serre) relation of minimal degree, which first appeared in \cite[Theorem~A]{CLW21b}.

\section{New proof of {\texti}Serre and Serre-Lusztig relations}\label{sec:newproof}

By reformulating the {\texti}Serre relations using the adjoint action, we can prove them using the representation theory of \(\Uqsl2\). Fix \(i \in \Isplit\), and let \(L(n)\) be the simple \(\Ui\)-module of highest weight \(q_i^n\). The following result was originally proved in \cite[Theorems 2.10, 3.6]{BeW18} in the special case where \(\vs_i = q_i^{-1}\). We will show that the general case follows.
\begin{lemma} [\cite{BeW18}]
\label{thm:BMinimalPolynomial}
    The element \(B_{i, \overline n}^{(n + 1)}\) annihilates \(L(n)\).  
\end{lemma}
\begin{proof}
    Let \(v_0\) be a highest weight vector in \(L(n)\). For \(k \geq 0\), we define \(v_k = F^{(k)}.v_0\). The elements \(v_0, v_1, \dots, v_n\) form a basis for \(L(n)\). By working in an extension of \(C(q)\), we may set \(z = \sqrt{q_i \vs_i}\). Let \(\widetilde{B_{i, \overline p}^{(n)}}\) be the {\texti}-divided powers with parameter \(q_i^{-1}\). We have \(\xi_z\of{\widetilde{B_{i, \overline{p}}^{(n)}}} = z^{-n} B_{i, \overline{p}}^{(n)}\). There is a uniquely defined linear isomorphism \(\xi_z : L(n) \to L(n)\) that sends \(\xi_z : F_i^{(k)} v_0 \mapsto z^{-k} F_i^{(k)} v_0\). By checking on generators, one can verify that for all \(u \in \Ui\) and all \(k \geq 0\), \(\xi_z\of{u.v_k} = \xi_z(u).\xi_z(v_k)\). In particular, by \cite[Theorems~ 2.10, 3.6]{BeW18}, \[
        B_{i, \overline n}^{(n + 1)}.v_k = z^n \xi_z\of{\widetilde{B_{i, \overline n}^{(n + 1)}}.\xi_z^{-1}(v_k)} = 0.
    \]
The lemma is proved. 
\end{proof}

We can now give a new short proof of the {\texti}Serre relations for {\texti}quantum groups. (In contrast, the original proof of $\imath$Serre relation in \cite{CLW21} was long and computational.)

\begin{theorem}[{\texti}Serre Relations]
    For \(i \in \Isplit\) and \(j \in \Iwhite\) with \(j \not= i\), we have
    \[
        \sum_{r + s = 1 - a_{ij}} (-1)^r B_{i, \overline{a_{ij}}}^{(s)}B_jB_{i, \overline 0}^{(r)} = 0.
    \]
\end{theorem}
\begin{proof}
    By Theorem~ \ref{cor:eq}, it suffices to show that
    \[
        \ad\of{B_{i, \overline{a_{ij}}}^{1 - a_{ij}}}\of{B_j \K_j} = 0.
    \]
    View \(\U\) as a \(\Ui\) module via the adjoint action. We can check \begin{align*}
        \ad\of{F_i}\of{T_\wblack(E_{\tau j})} 
        &= (F_i T_\wblack(E_{\tau j}) - T_\wblack(E_{\tau j}) F_i)\K_i \\
        &= T_\wblack(F_i E_{\tau j} - E_{\tau j} F_i)\K_i 
        = 0,
    \end{align*}
    so \(T_\wblack(E_{\tau j})\) is a lowest weight vector. For any \(u \in \U\), \begin{align*}
        \ad\of{E_i}\of{T_\wblack(u)} 
        = E_i T_\wblack(u) - \K_i T_\wblack(u) \K_i^{-1} E_i \\
        = T_\wblack\of{E_i u - \K_i u \K_i^{-1} E_i} 
        = T_\wblack\of{\ad\of{E_i}(u)},
    \end{align*}
    so in particular, by \eqref{eq:eqSerre}, we have \[
        \ad\of{E_i^{1 - a_{ij}}\of{T_\wblack\of{E_{\tau j}}}} = T_\wblack\of{\ad\of{E_i^{1 - a_{ij}}\of{E_{\tau j}}}} = 0.
    \]
    Therefore, the submodule generated by \(T_\wblack(E_{\tau j})\) must be a finite dimensional simple module. Since \(T_\wblack(E_{\tau j})\) is a lowest weight vector of weight \(q_i^{a_{ij}}\), this module has highest weight \(q_i^{-a_{ij}}\). Therefore, \[
        \ad\of{B_{i, \overline{a_{ij}}}^{(1 - a_{ij})}}\of{T_\wblack(E_{\tau j})} = 0.
    \]
    Similarly, we can check \[
        \ad\of{E_i}\of{F_j \K_j} = E_i F_j \K_j - \K_i F_j \K_j \K_i^{-1} E_i = (E_i F_j - F_j E_i)\K_j = 0,
    \]
    so \(F_j \K_j\) is a highest weight vector. By the \(q\)-Serre relations, it generates a finite-dimensional simple module of highest weight \(q_i^{-a_{ij}}\). Therefore, \[
        \ad\of{B_{i, \overline{a_{ij}}}^{(1 - a_{ij})}}\of{F_j \K_j} = 0.
    \]
    Finally, we may conclude \begin{align*}
       & \ad\of{B_{i, \overline{a_{ij}}}^{(1 - a_{ij})}}\of{B_j \K_j} 
        \\
        & = \ad\of{B_{i, \overline{a_{ij}}}^{(1 - a_{ij})}}\of{F_j \K_j} + \vs_j \ad\of{B_{i, \overline{a_{ij}}}^{(1 - a_{ij})}}\of{T_\wblack(E_{\tau j})} 
        = 0.
    \end{align*}
The theorem is proved. 
\end{proof}

For the Serre-Lusztig relations, we will need a stronger variant of Lemma \ref{thm:BMinimalPolynomial}.
\begin{lemma}\label{thm:strongBMinimalPolynomial}
    For any \(n, k \geq 0\), \(B_{i, \overline{kn}}^{(kn + 1)}\) annihilates \(L(n)^{\otimes k}\).
\end{lemma}
\begin{proof} 
    We may assume \(n, k \geq 1\). It is well known that all the simple summands of \(L(n)^{\otimes k}\) must be of the form \(L(kn - 2t)\) for some \(t \geq 0\). Since \(B_{i, \overline{kn}}^{(kn - 2t + 1)}\) divides \(B_{i, \overline{kn}}^{(kn + 1)}\), by Lemma \ref{thm:BMinimalPolynomial}, \(B_{i, \overline{kn}}^{(kn + 1)}\) annihilates each simple summand of \(L(n)^{\otimes k}\).
    Therefore, \(B_{i, \overline{kn}}^{(kn + 1)}\) annihilates the whole module \(L(n)^{\otimes k}\).
\end{proof}

Additionally, recall that (see, e.g.,  \cite{JL92}) for any Hopf algebra \(H\) with comultiplication \(\Delta\), if \(x, y, z \in H\), and \(\Delta(x) = \sum_{t = 1}^k x_t \otimes x^t\), then 
\[
    \ad(x)(yz) = \sum_{t = 1}^k \ad(x_t)(y)\ad(x^t)(z).
\]
Equivalently, the multiplication map \(\nabla : H \otimes H \to H\) is an \(H\)-module homomorphism, where \(H\) is viewed as an \(H\)-module via \(\ad\) and \(H \otimes H\) is viewed as an \(H\)-module via \((\ad \otimes \ad)\Delta\).

We now give a new conceptual proof of the Serre-Lusztig relations of minimal degree for {\texti}quantum groups. (The original proof of these Serre-Lusztig relations in \cite{CLW21b} was long and computational.)

\begin{theorem}[Serre-Lusztig relations of minimal degree for {\texti}quantum groups]\label{thm:SerreLusztig}
    For all \(i \in \Isplit\), \(n \geq 0\), and \(j \in \Iwhite\) with \(j \not= i\), \[
        \sum_{r + s = 1 - na_{ij}} (-1)^r B_{i, \overline{n a_{ij}}}^{(s)}B_j^n B_{i, \overline 0}^{(r)} = 0.
    \]
\end{theorem}
\begin{proof}
By Proposition \ref{thm:iSerreLusztigAdjoint}, it suffices to show \[
    \ad\of{B_{i, \overline{n a_{ij}}}^{1 - na_{ij}}}\of{B_j^n \K_j^n} = 0.
\]
View \(\U\) as a \(\Uqisl2\) module via the adjoint action of \(\Ui\). In the proof of the {\texti}Serre relations, we showed that the modules generated by \(T_\wblack(E_{\tau j})\) and \(F_j \K_j\) are both isomorphic to \(L(-a_{ij})\). Therefore, there is a \(\Uqisl2\)-module homomorphism \(\psi: L(-a_{ij}) \oplus L(-a_{ij}) \to \U\) whose image contains both \(T_\wblack(E_{\tau j})\) and \(F_j \K_j\). In particular, for every \(t \in \Z\), \(\xi_{q_j}^{t}(B_j)\K_j\) is contained in the image of \(\psi\). Let \(\nabla: \U^{\otimes k} \to \U\) be the multiplication map. Let \(\Psi\) be the composition \[
    \bigoplus_{t = 1}^{2^n}L(-a_{ij})^{\otimes n} \cong \of{L(-a_{ij}) \oplus L(-a_{ij})}^{\otimes n} \xrightarrow{\psi^{\otimes n}} \U^{\otimes n} \xrightarrow{\nabla} \U.
\]
Then \(\Psi\) is a homomorphism of \(\Uqisl2\) modules, and the image of \(\Psi\) contains 
\[
    B_j \K_j \xi_{q_j}^{-2}\of{B_j}\K_j \cdots \xi_{q_j}^{-2n + 2}\of{B_j} \K_j = B_j^n \K_j^n.
\]
But by Lemma \ref{thm:strongBMinimalPolynomial}, \(B_{\overline{n a_{ij}}}^{1 - na_{ij}}\) annihilates \(\bigoplus_{t = 1}^{2^n}L(-a_{ij})^{\otimes n}\), so \[
    \ad\of{B_{\overline{n a_{ij}}}^{1 - na_{ij}}}\of{B_j^n \K_j^n} = 0.
\]
The theorem is proved. 
\end{proof}

In the above proof, there was nothing special about the fact that the same \(j\) was repeated \(n\) times. The same argument can be applied to prove the following stronger version, which seems new.
\begin{theorem}
    Let \(i \in \Isplit\), and \(j_1, j_2, \dots, j_k \in \Iwhite \setminus \{i\}\). Set $
        n = \sum_{t = 1}^k a_{ij_k}.$
    Then we have
    \[
        \sum_{r + s = 1 - n} (-1)^r B_{i, \overline n}^{(s)} B_{j_1} \cdots B_{j_k} B_{i, \overline 0}^{(r)} = 0.
    \]
\end{theorem}
\begin{proof}
    We first need to show that \(B_{i, \overline n}^{(n + 1)}\) annihilates \(L(-a_{ij_1}) \otimes \cdots \otimes L(-a_{ij_k})\). It is clear that all simple summands of the module \(L(-a_{ij_1}) \otimes \cdots \otimes L(-a_{ij_k})\) have highest weights of the form \(q_i^{n - 2t}\) for some \(t \in \Z\). Therefore, by Lemma~\ref{thm:BMinimalPolynomial}, \(B_{i, \overline n}^{(n + 1)}\) annihilates each of these summands, so it annihilates the whole module. The rest of the proof is omitted, as it is identical to the proof of Theorem \ref{thm:SerreLusztig}.
\end{proof}

% \begin{remark}
%    Theorem \ref{thm:sformula} and its consequences remain valid for general {\texti}quantum groups (cf. \cite{Ko14, CLW21}) as long as \(T_{w_{\bullet}}(E_{\tau i}) = E_i\). In the quasi-split case, this is equivalent to \(\tau i = i\). In particular, the methods of this paper give a new proof of \cite[(3.9)]{CLW21} in the quasi-split setting. Similarly, Proposition \ref{thm:iSerreLusztigAdjoint} applies to general {\texti}quantum groups as long as \(T_{w_{\bullet}}(E_{\tau i}) = E_i\).
% \end{remark}

\vspace{4mm}

%\noindent {\bf Note.} The author has no conflicts of interest to declare.

\end{document}